\documentclass{article}
\usepackage{latexsym}
\usepackage{hyperref}
\usepackage[all]{xy}

\usepackage{amsmath}
\usepackage{xypic}

\usepackage{amsmath}
\usepackage{amssymb}
\usepackage{amsfonts}
\usepackage{amsthm}
\usepackage{mathrsfs}
\usepackage[utf8]{inputenc}

\usepackage{multicol}

\DeclareUnicodeCharacter{00A0}{ }

\DeclareFontFamily{U}{wncy}{}
\DeclareFontShape{U}{wncy}{m}{n}{<->wncyr10}{}
\DeclareSymbolFont{mcy}{U}{wncy}{m}{n}
\DeclareMathSymbol{\Sh}{\mathord}{mcy}{"58}

\newtheorem{theo}{Theorem}

\newtheorem{coro}[theo]{Corollary}
\newtheorem{lemm}[theo]{Lemma}

\newtheorem*{theoUn}{Theorem}
\theoremstyle{definition}
\newtheorem{defi}[theo]{Definition}

\newtheorem{rema}[theo]{Remark}

\newtheorem{ques}[theo]{Question}

\newcommand{\cA}{\mathcal{A}}
\renewcommand{\AA}{\mathbb{A}}
\newcommand{\bbA}{\mathbb{A}}

\newcommand{\OO}{\mathcal{O}}

\newcommand{\PP}{\mathbb{P}}

\DeclareMathOperator{\Spec}{Spec}

\DeclareMathOperator{\codim}{codim}

\newcommand{\oc}{\overline{\square}}
\newcommand{\TT}{\mathfrak{T}}

\newcommand{\VV}{\mathfrak{V}}

\newcommand{\XX}{\mathfrak{X}}
\newcommand{\YY}{\mathfrak{Y}}
\newcommand{\EE}{\mathfrak{E}}
\newcommand{\ZZ}{\mathfrak{Z}}

\newcommand{\PreShv}{\mathsf{PreShv}}
\newcommand{\Shv}{\mathsf{Shv}}

\newcommand{\Cor}{\mathsf{Cor}}
\newcommand{\uMCor}{\mathsf{\underline{M}Cor}}

\newcommand{\Nis}{\mathsf{Nis}}

\newcommand{\DM}{\operatorname{\mathbf{DM}}}

\newcommand{\uMDM}{\operatorname{\mathbf{\underline{M}DM}}}

\newcommand{\MV}{\mathsf{MV}}
\newcommand{\CI}{\mathsf{CI}}

\newcommand{\ol}{\overline}
\newcommand{\bcube}{{\ol{\square}}}
\newcommand{\eff}{{\operatorname{eff}}}

\newcommand{\Xinf}{X_\infty}
\newcommand{\Zinf}{Z_\infty}

\title{Smooth blowup square for motives with modulus}

\author{Shane Kelly, Shuji Saito}

\begin{document}

\maketitle

\begin{abstract}
In this self-contained paper we prove that Voevodsky's smooth blowup triangle of motives generalises to a smooth blowup triangle of motives with modulus.
\end{abstract}

%
%
%
%
%

\section{Introduction}

In \cite[Prop.3.5.2]{voe3}, Voevodsky proves that if $Z \to X$ is a (regular) closed immersion of smooth $k$-varieties, then there is a distinguished triangle 
\[ M(E) \to M(Bl_XZ) \oplus M(Z) \to M(X) \to M(E)[1] \]
associated to the blowup $Bl_XZ$ where $E$ is the exceptional divisor. In this article, we prove a modulus version of this result which specialises to Voevodsky's under the canonical ``interior'' functor $\uMDM^{\eff}(k) \to \DM^\eff(k)$.

Our situation is the following: $\ol{X}$ is a smooth $k$-variety, and $X_\infty$ an effective Cartier divisor on $\ol{X}$ with support strict normal crossings. We have a closed immersion of smooth varieties $\ol{Z} \to \ol{X}$ which is transverse to $X_\infty$ (see Def.\ref{defi:lst} for the precise meaning of transverse). Let $\ol{X'}$ be the blowup of $\ol{X}$ in $\ol{Z}$ with exceptional divisor $\ol{E}$, and let $X'_\infty, E_\infty, Z_\infty$ be the respective pullbacks of $X_\infty$.

\begin{theoUn}
There is a canonical distinguished triangle
\[ M(\ol{E}, E_\infty) \to M(\ol{Z}, Z_\infty) \oplus M(\ol{X'}, X'_\infty) \to M(\ol{X}, X_\infty) \to M(\ol{E}, E_\infty)[1] \]
in $\uMDM^{\eff}$ (as defined in Definition~\ref{defi:MDM}).
\end{theoUn}

There are at least two obvious extensions of this result which we do not deal with, partly to keep this paper self-contained, but also because otherwise it might never appear (it has been sitting in the authors' drawer since 2017).

\emph{Splitting.} In the presence of a projective bundle theorem, the proof of \cite[Prop.3.5.3]{voe3} (with $\AA^1$ replaced with $\bcube$) would show that this triangle has a canonical splitting.

\emph{Cd structures.} 
If $k$ has characteristic zero, or more generally satisfies a strong resolution of singularities hypothesis, we expect that the class of squares of the form \eqref{blowupsquare} (on page~\pageref{blowupsquare}) together with those in the class $\MV$ form a bounded, complete, regular cd structure on a suitable subcategory of $\uMCor(k)$ in the sense of \cite{voe2}. We leave this question for another time. We point out only that if one considers the associated topology on $\uMCor(k)$, the existence of a left adjoint to the inclusion $\Shv(\uMCor(k)) \to \PreShv(\uMCor(k))$ seems to be subtle, due to the fact that (Left properness) allows non-finiteness, cf.Rem.\ref{rema:leftProperRemark}.

This paper is self-contained in the sense that it doesn't use any results from the 2018 Kahn, Saito, Yamazaki preprint ``Motives with Modulus''. Our definition of $\uMDM^{\eff}(k)$ is slightly non-standard, but our definition is expected to agree with the standard one cf.Remark~\ref{rema:MDMdefs}.

The main result of this paper is applied by Matsumoto in \cite{mat} and its sequel to produce various generalisations of Voevodsky's Gysin triangle.

\section{Basic definitions}

Following tradition, we always work over a perfect field $k$.

In this section, \emph{after} defining all the requisite terms, we \emph{will} define $\uMDM^{\eff}$ as:

\begin{defi} \label{defi:MDM}
The category $\uMDM^{\eff}(k)$ is the Verdier localisation
\[ \uMDM^{\eff}(k) \stackrel{def}{=} \frac{D(\PreShv(\uMCor(k)))}{\left \langle \MV, \CI \right \rangle} \]
of the derived category of the category of presheaves on the additive category $\uMCor(k)$, Def.~\ref{defi:Mcor}, with respect to the two classes of complexes $\MV$ and $\CI$, Def.~\ref{defi:MVCI}. We write 
\[ M: \uMCor(k) \to \uMDM^{\eff} \]
for the functor induced by the Yoneda embedding $\uMCor(k) \to \PreShv(\uMCor(k))$.
\end{defi}

\begin{rema} \label{rema:MDMdefs}
In the upcoming Kahn, Saito, Yamazaki paper ``Motives with Modulus, II'' the category $\uMDM^{\eff}(k)$ will be defined as using Nisnevich sheaves as $D(\Shv_{\Nis}(\uMCor(k))) / \langle \CI \rangle$. This latter definition is expected to produce the same category as Definition~\ref{defi:MDM} (cf.\cite{voe2}). We do not use Nisnevich sheaves anywhere in this paper. 
\end{rema}

We begin with $\uMCor(k)$. This is a generalisation of Voevodsky's category $\Cor(k)$ which incorporates the notion of a modulus.

\begin{defi}[{Kahn, Saito, Yamazaki}] \label{defi:Mcor}
Objects of the category $\uMCor(k)$ are pairs $\XX = (\ol{X}, X_\infty)$ where:
\begin{description}
 \item[{$\ol{X}$}] is a separated $k$-scheme of finite type which is locally integral, and
 \item[{$X_\infty$}] is an effective Cartier divisor on $\ol{X}$ such that 
 \item[{$X^\circ \stackrel{def}{=} \ol{X} - X_\infty$}] is smooth.
\end{description}
Such pairs are called \emph{modulus pairs}. Given two modulus pairs $\XX, \YY$ the hom group
\[ \hom_{\uMCor}(\XX, \YY) \subseteq \hom_{\Cor}(X^\circ, Y^\circ) \]
is the subgroup of left proper, admissible correspondences. That is, it is the free abelian group associated to the set of closed integral subschemes $Z \subseteq X^\circ \times Y^\circ$ such that 
\begin{description}
 \item[{(Correspondenceness)}] $Z \to X^\circ$ is finite and dominates a connected component of $X^\circ$,
 \item[{(Admissibility)}] $p^*X_\infty \geq q^*Y_\infty$ were $p, q: {\ol{Z}}^N \to \ol{X}, \ol{Y}$ are the canonical morphisms from the normalisation ${\ol{Z}}^N$ of the closure $\ol{Z}$ of $Z$ in $\ol{X} \times \ol{Y}$.
 \item[{(Left properness)}] $\ol{Z} \to \ol{X}$ is proper, and
 \end{description}
The category is $\uMCor(k)$ is additive; $(\ol{X}, X_\infty) \oplus (\ol{Y}, Y_\infty) = (\ol{X} \amalg \ol{Y}, X_\infty + Y_\infty)$. It is also equipped with a symmetric mono{\"i}dal structure, given on objects by
\[ (\ol{X}, X_\infty) \otimes (\ol{Y}, Y_\infty) = (\ol{X} \times \ol{Y}, \ol{X} {\times} Y_\infty +  X_\infty {\times} \ol{Y}) \]
On morphisms it is the same as the product structure on Voevodsky's category $\Cor(k)$. In other words, the canonical faithful functor $(-)^\circ: \uMCor(k) \to \Cor(k)$ is mono{\"i}dal. However, we use the tensor structure almost exclusively as a notational convenience. The most complicated correspondences that we will apply $\otimes$ to are graphs of morphisms of schemes.

In the theory of motives with modulus the \emph{cube} object
\[ \bcube = (\PP^1, \{\infty\}) \]
takes the r{\^o}le of $\AA^1$.
\end{defi}

\begin{rema} \label{rema:leftProperRemark}
Note that (Left properness) allows non-finite morphisms; we are allowed to blowup inside the modulus. Such blowups are isomorphisms in $\uMCor$. This is a requisite for $\bcube$ to have the structure of an interval object in $\uMCor$ in the sense of Voevodsky, however it makes the sheaf theory more subtle.

Since we do not use sheaves in this article, this does not concern us.

What does concern us however, is the fact that blowups inside the modulus are isomorphisms. We will use this fact in Lemma~\ref{lemm:1} to show that $(\PP^{n+1}, \PP^n)$ is contractible in $\uMDM^\eff(k)$. This is the only place where this fact is used.
\end{rema}

\begin{defi} \label{defi:MVCI}
We consider the following complexes in the additive category $\uMCor(k)$.
\begin{description}
 \item[{$\MV$}] is the class of complexes of the form
 \[ [ (\ol{W}, W_\infty) \to (\ol{V}, V_\infty) \oplus (\ol{U}, U_\infty) \to (\ol{X}, X_\infty) ] \]
induced by cartesian squares of $k$-schemes
\[ \xymatrix{
\ol{W} \ar[r] \ar[d] \ar[dr]_h & \ol{V} \ar[d]^f \\
\ol{U} \ar[r]_{j} & \ol{X}
} \]
where $j$ is an open immersion, $f$ is étale, and $f$ induces an isomorphism over $\ol{X} \setminus \ol{U}$. We require the morphisms of modulus pairs to be \emph{minimal} in the sense that $U_\infty = j^*X_\infty$, $V_\infty = f^*X_\infty$, and $W_\infty = h^*W_\infty$.

 \item[{$\CI$}] is the class of complexes of the form
 \[ [\XX \otimes \bcube \to \XX] \]
for $\XX \in \uMCor(k)$.
\end{description}
\end{defi}

Killing the complexes in $\MV$ leads to the obvious locality properties one might expect, such as the following.

\begin{lemm} \label{lemm:MVcubeBundle}
Let $(\ol{X}, X_\infty) \to (\ol{Y}, Y_\infty)$ be a $\bcube$-bundle. I.e., a morphism of modulus pairs induced by a morphism of schemes $f: \ol{X} \to \ol{Y}$ such that there exists an open Zariski covering $\{V_i \to \ol{Y}\}_{i \in I}$ and isomorphisms (compatible with the morphisms to $(V_i,V_i \cap Y_\infty)$)
\[ (U_i, U_i \cap X_\infty) \cong (V_i, V_i \cap Y_\infty) \otimes \bcube \]
where $U_i = f^{-1}V_i$. Then $M(\ol{X}, X_\infty) \cong M(\ol{Y}, Y_\infty)$.
\end{lemm}

\begin{proof}
Since $\ol{Y}$ is quasicompact we can assume $I$ is finite. By induction on the size of $I$ we can assume that $M(U, U \cap X_\infty) \to M(V, V \cap Y_\infty)$ is an isomorphism for all $U, V$ of the form $U = \cup_{i \in J} U_i, V = \cup_{i \in J} V_i$ and $J \subsetneq I$. Choose some $i \in I$, set $J = I \setminus \{i\}$, and consider the diagram
\[ \xymatrix{
(U', U'_\infty) \ar[r] \ar[d] & (U, U_\infty) \oplus (U_i, U_{i\infty}) \ar[r] \ar[d] & (\ol{X}, X_\infty) \ar[d] \\
(V', V'_\infty) \ar[r] & (V, V_\infty) \oplus (V_i, V_{i\infty}) \ar[r] & (\ol{Y}, Y_\infty)
} \]
where the divisors are the obvious restrictions of $Y_\infty$ (resp. $X_\infty$), $U = \cup_{i \in J} U_i, V = \cup_{i \in J} V_i$, and $U' = U \cap U_i, V' = V \cap V_i$.

The left and middle verticle morphisms become isomorphisms in $\uMDM^{\eff}$ by hypothesis (inductive, and the one in the statement). Hence, the total complex of the left square is zero in $\uMDM^{\eff}$. On the other hand, the rows become zero in $\uMDM^{\eff}$ as they are in the class $\MV$. Hence, the total complex also becomes zero in $\uMDM^{\eff}$. So $[(\ol{X}, X_\infty) \to (\ol{Y}, Y_\infty)]$ is isomorphic in $\uMDM^{\eff}$ to the total complex of the left square, which we have seen to be zero.
\end{proof}

\section{Log smooth modulus pairs}

In Voevodsky's theory, one of the main uses of the distinguished Nisnevich squares is that, Nisnevich locally, (regular) closed immersions $Z \to X$ of smooth varieties are isomorphic to zero sections $Z \to \AA^c_Z$. We will use the Nisnevich condition in this way. However, we must isolate what we mean by a ``regular'' closed immersion of ``smooth'' modulus pairs.

\begin{defi} \label{defi:lst}
A modulus pair $(\ol{X}, X_\infty)$ is \emph{log smooth} if the total space $\ol{X}$ is smooth, and the support $|X_\infty|$ of the modulus is a strict normal crossings divisor. In other words, for every $x \in \ol{X}$, there exists a Zariski open neighbourhood $x \in U \stackrel{\iota}{\hookrightarrow}{X}$ and an étale morphism $q: U \to \AA^d = \Spec(k[T_1, \dots, T_d])$ such that $|\iota^{-1}X_\infty| = q^{-1}(\{T_1\dots T_s = 0\})$.

A morphism $i: (\ol{Z}, Z_\infty) \to (\ol{X}, X_\infty)$ of log smooth modulus pairs is said to be \emph{transversal} if
\begin{enumerate}
 \item $i$ is induced by a closed immersion $\ol{Z} \to \ol{X}$,
 \item $Z_\infty = i^{-1}X_\infty$, and
 \item for every $z \in \ol{Z}$ there exists an open neighbourhood $z \in U \stackrel{\iota}{\hookrightarrow}{X}$ and an étale morphism $q: U \to \AA^d = \Spec(k[T_1, \dots, T_d])$ such that $|\iota^{-1}X_\infty| = q^{-1}(\{T_1\dots T_s = 0\})$ and $\ol{Z} \cap U = q^{-1}(\AA^{d-c} \times \{0, \dots, 0\})$ with $s \leq d-c$.
\end{enumerate}
\end{defi}

\begin{lemm} \label{lemm:cover}
Let $(\ol{Z}, Z_\infty) \to (\ol{X}, X_\infty)$ be a transversal morphism of log smooth modulus pairs. Then there exists an open covering $\{j_i: U_i \to \ol{X}\}_{i = 0}^n$, and \'etale morphisms $U_i \stackrel{pr_1}{\leftarrow} U'_i \stackrel{pr_2}{\to} Z_i \times \bbA^c$ for $1 \leq i \leq n$ such that $U_0 = \ol{X} \setminus \ol{Z}$ and 
\[ Z_i \cong pr_1^{-1}(Z_i) = pr_2^{-1}(Z_i {\times} \{0, \dots, 0\})\;\text{ and }\;
pr_1^{-1}(U_{i \infty}) = pr_2^{-1}(Z_{i \infty} \times \bbA^c),\]
where $Z_i, U_{i \infty}, Z_{i \infty}$ are the intersections of $\ol{Z}, \Xinf$, and $\ol{Z} \cap \Xinf$ with $U_i$.

\end{lemm}

\begin{proof}
We follow an argument in the proof of \cite[Lem.3.2.28]{mv99} with a small modification. By the definition of ``transversal'', every point $x \in \ol{Z}$ admits an open neighbourhood $x \in U \subseteq \ol{X}$ equipped with an étale morphism $q: U\to \bbA^{r+c}=\Spec k[T_1,\dots,T_{r+c}]$ such that $Z = q^{-1}(\bbA^r {\times} \{0,\dots,0\})$ and $|U_\infty|=q^{-1}(\{T_1\dots T_s=0\})$ with $s\leq r$ and $c=\codim_{U}(Z)$, where $Z = \ol{Z} \cap U$ and $U_\infty = X_\infty \cap U$. Define $\Gamma= U \times_{\bbA^{r+c}} (Z \times \bbA^c)$, where the right morphism comes from the composition $Z \to X \to \bbA^r {\times} \bbA^c \to \bbA^r$. 
Then 
\[ \Gamma\times_{\bbA^{r+c}} (\bbA^r\times\{0,\dots,0\}) \cong Z \times_{\bbA^r} Z.\]
Since $Z \to \bbA^r$ is \'etale (or rather, because it is unramified), the above is a disjoint union of the diagonal $Z \hookrightarrow Z \times_{\bbA^r} Z$
and a closed subscheme $\Sigma\subset Z \times_{\bbA^r} Z$.
Put $U'=\Gamma{-}\Sigma$ with projections $pr_1:U'\to U$ and $pr_2 :U' \to Z \times\bbA^c$.
By the construction, 
\[ pr_1^{-1}(Z) \cong Z,
\;\; 
pr_2^{-1}(Z \times \{0,\dots,0\}) \cong Z \times \{0,\dots,0\},\;\; %
pr_1^{-1}(U_\infty) =  pr_2^{-1}( \Zinf \times \bbA^c), \]
where $\Zinf = U_\infty \cap Z$. This implies the lemma.
\end{proof}

\begin{coro} \label{coro:etNisLoc}
Let $(\ol{Z}, Z_\infty) \to (\ol{X}, X_\infty)$ be a transversal morphism of log smooth modulus pairs. Then there exists an open covering $\{U_i \to \ol{X}\}_{i = 0}^n$, and \'etale morphisms %
$U_i \stackrel{pr_1}{\leftarrow} U'_i \stackrel{pr_2}{\to} \ol{Z} {\times} \PP^c$ %
for $1 \leq i \leq n$ such that $U_0 = \ol{X} \setminus \ol{Z}$, and 
\[ Z_i \cong pr_1^{-1}(Z_i) = pr_2^{-1}(Z_i {\times} \{0, \dots, 0\})\;\text{ and }\;
pr_1^{-1}(U_{i \infty}) = pr_2^{-1}(Z_{i \infty} \times \PP^c),\]
where $Z_i, U_{i \infty}, Z_{i \infty}$ are the intersections of $\ol{Z}, \Xinf$, and $\ol{Z} \cap \Xinf$ with $U_i$, and $\{0, \dots, 0\} \in \PP^c$ is the point orthogonal to $\PP^{c-1} \subseteq \PP^c$ 
\end{coro}

\begin{proof}
Take the cover from Lemma~\ref{lemm:cover}, and compose $pr_2$ with the inclusion $\bbA^c   \cong \PP^c \setminus \PP^{c-1} \to \PP^c$. 
\end{proof}

\section{Toric invariance}

Voevodsky's proof that the smooth blowup triangle in $\DM^{\eff}$ is distinguished roughly has two main steps.
\begin{enumerate}
 \item Nisnevich locally, blowups of regular immersions of smooth schemes look like the product of the closed subscheme $Z$ with the blowup of an affine space in the origin,
 \[ \xymatrix{
\PP^{c-1} \ar[r] \ar[d] & Bl_{\AA^c} \{0\} \ar[d] \\
\{0\} \ar[r] & \AA^c
 } \]

 \item By $\AA^1$-invariance, the two horizontal morphisms are isomorphisms.
\end{enumerate}

The following lemma is our version of the second step.

\begin{lemm} \label{lemm:1}
For all $\XX \in \uMCor$, we have  
\[ M(\XX \otimes (\PP^{n+1}, \PP^{n})) \cong M(\XX) \;\;\text{in }\uMDM^{\eff}.\]
Moreover, consider $\pi: P \to \PP^{n+1}$ the blowup of a point $x \in \PP^n$ with exceptional divisor $E$ and let $H$ be the strict transform of $\PP^{n}$. If $x \not\in \PP^{n}$, then 
\[ M(\XX \otimes(E, \varnothing)) \cong M(\XX \otimes (P, H)) \]
for all $\XX \in \uMCor$, and if $x \in \PP^{n}$, then
\[ M(\XX \otimes (P, \pi^{-1}\PP^{n})) \cong M(\XX)  \]
for all $\XX \in \uMCor$.
\end{lemm}

\begin{proof}
The proof is by induction on $n$. It is true for $n = 0$ by definition of $\uMDM^{\eff}$. 
In the notation of the statement, it follows from the definition of $\uMCor$ that if $x \in \PP^{n}$, then $(\PP^{n+1}, \PP^{n}) \cong (P, \pi^{-1}\PP^{n})$ in $\uMCor$, so it suffices to show the second two claimed isomorphisms.

Recall that $P$ is isomorphic to $\PP(\OO_{\PP^{n}}(-1) \oplus \OO_{\PP^{n}})$ and in particular, there is a canonical projection $P \to \PP^{n}$ making $P$ a $\PP^1$-bundle over $\PP^{n}$, which maps $E$ isomorphically to $\PP^{n}$. 

Suppose first that $x \in \PP^{n}$. Then under the identification $E \cong \PP^{n}$, the divisor $H$ is $\pi^{-1}(\pi(H \cap E))$. More importantly, the induced map $(P, \pi^{-1}\PP^n) \to (E, \pi(E \cap H)) \cong (\PP^{n}, \PP^{n-1})$ is a $\oc$-bundle, cf. Lemma~\ref{lemm:MVcubeBundle}. Consequently, Lemma~\ref{lemm:MVcubeBundle} implies $M(\XX \otimes (P, \pi^{-1}\PP^n)) \cong M(\XX \otimes (\PP^{n}, \PP^{n-1}))$, and so by the inductive hypothesis, we find that $M(\XX \otimes (P, \pi^{-1}\PP^n)) \cong M(\XX)$, as desired. On the other hand, if $x \not\in \PP^{n}$, then $(P, H) \to (E, \varnothing)$ is a $\oc$-bundle, and the same argument produces the other desired isomorphism.
\end{proof}

\begin{ques}
If $\TT \in \uMCor$ is such that $\ol{T}$ is a toric variety and $|T_\infty| \to \ol{T}$ is an inclusion of toric varieties. When do we have $M(\TT) \cong M(\Spec(k), \varnothing)$?

If $T_\infty = T_\infty + T'_\infty$, when do we have $M(T_\infty, \varnothing) \cong M(\ol{T}, T_\infty + T_\infty')$?
\end{ques}

\section{Smooth blowups}\label{smoothblowup}

Let $\ZZ \to \XX$ be a transversal morphism of log smooth modulus pairs, Def.~\ref{defi:lst}. %
Let $\pi: \ol{X'} \to \ol{X}$ be the blowup of $\ol{X}$ in $\ol{Z}$ with exceptional divisor $j:\ol{E} \to \ol{X'}$. %
Put $\XX'=(\ol{X}',\pi^{-1}(\Xinf))$ and
$\ZZ=(\ol{Z},\Zinf)$ with $\Zinf=\ol{Z}\cap \Xinf$ and $\EE=(\ol{E},\pi^{-1}(\Zinf))$.

We are interested in the square
\begin{equation}\label{blowupsquare}
\xymatrix{
\EE  \ar[r] \ar[d] & \XX' \ar[d] \\
\ZZ \ar[r] & \XX.
} 
\end{equation}

Consider the following statement.

\begin{enumerate}
 \item[{$(SBU)_{\ZZ \stackrel{i}{\to}\XX}$}] The complex $\EE \to \ZZ \oplus \XX' \to \XX$ is isomorphic to zero%
 \footnote{This implies that the square \eqref{blowupsquare} becomes homotopy cartesian in $\uMDM^{\eff}$ in the sense of \cite[Def.1.4.1]{nee} but a priori, is stronger. We work with the stronger statement because the nine lemma (cf. proof of Lemma~\ref{lemm:2}) and 2-out-of-3 property (cf. proof of Lemma~\ref{lemm:3}) are much easier in this setting. Indeed, we don't even know if these two facts, as we want them stated, are true in an abstract triangulated category.
} %
  in $\uMDM^{\eff}$.
 \end{enumerate}

\begin{lemm} \label{lemm:2}
If there exists an open covering $\{U_i \to \ol{X}\}_{i \in I}$ such that 
$(SBU)_{\ZZ_J \stackrel{}{\to} \XX_J}$ is true for every $J \subsetneq I$ where $U_J = \cap_{i \in J} U_i$ and $\XX_J=(U_J, X_\infty \cap U_J)$ and
$\ZZ_J=({Z}\cap U_J,\Zinf\cap U_J)$. Then $(SBU)_{\ZZ \stackrel{i}{\to}\XX}$ is true.
\end{lemm}

\begin{proof}
As $\ol{X}$ is quasicompact, we can assume that $I$ is finite, say of size $n$. By induction on $n$, it suffices to consider the case $n = 2$. Consider the diagram
\[ \xymatrix{
\EE_{12} \ar[r] \ar[d] & \ZZ_{12} \oplus \XX'_{12} \ar[r] \ar[d] & 
\XX_{12} \ar[d] \\
\EE_1 \oplus \EE_2 \ar[r] \ar[d] & 
\biggl (\ZZ_1 \oplus \XX'_1 \biggr ) \oplus 
\biggl (\ZZ_2 \oplus \XX'_2 \biggr ) \ar[r] \ar[d] & 
\XX_1 \oplus \XX_2  \ar[d] \\
\EE \ar[r] & \ZZ \oplus \XX' \ar[r] & \XX
} \]
where $\XX'_J, \EE_J$ are the obvious analogues of $\XX_J, \ZZ_J$. Since the columns belong to the class $\MV$, they are zero in $\uMDM^{\eff}$, so the total complex is zero as well. By hypothesis, the top two rows are zero in $\uMDM^{\eff}$, so the lower row is isomorphic to the total complex. But we have just seen that this latter is zero, and therefore so is the former.
\end{proof}

\begin{lemm} \label{lemm:3}
Let $q: V \to \ol{X}$ be an \'etale morphism such that $q^{-1}{\ol{Z}} = {\ol{Z}}$. Define $\VV = (V, q^{-1}X_\infty)$. Then $(SBU)_{\ZZ \to \XX}$ is true if and only if $(SBU)_{\ZZ{\to}\VV}$ is true.
\end{lemm}

\begin{proof}
This follows from the 2-out-of-3 property of homotopy cartesian squares of chain complexes, cf. Lemma~\ref{lemm:2o3}. 
Let $U = \ol{X} - Z$, and consider the following squares
\[ \xymatrix{
V {\times_{\ol{X}}} U \ar[r] \ar[d] \ar@{}[dr]|{(A)} & U \ar[d] && %
\ol{E} \ar[r] \ar[d] \ar@{}[dr]|{(C)} & \ol{Z} \ar[d] %
\\
V {\times_{\ol{X}}} {\ol{X}}' \ar[r] \ar[d] \ar@{}[dr]|{(B)} & {\ol{X}}' \ar[d] && %
{\ol{X}}' {\times_{\ol{X}}} V \ar[r] \ar[d] \ar@{}[dr]|{(D)} & V \ar[d] 
\\
V \ar[r] & {\ol{X}} && %
{\ol{X}}' \ar[r] & {\ol{X}}
} \]
The square (A), and the outer square of (A) and (B) are distinguished Nisnevich squares, so their associated objects in $\uMDM^{\eff}$ are zero. Hence, the same is true of (B), by Lemma~\ref{lemm:2o3}. 
But (B) and (D) are isomorphic, so (D) also gives rise to a zero object of $\uMDM^{\eff}$. Therefore, again by the 2-out-of-3 property, (C) gives rise to a zero object if and only if the outer square of (C) and (D) does. Since flatness of $V \to {\ol{X}}$ implies that the pullback square (D) is also a strict transform square, (C) giving rise to a zero object is precisely $(SBU)_{\ZZ{\to}\VV}$.
\end{proof}

\begin{theo}\label{theo;smoothblowup}
The statement $(SBU)_{\ZZ \to \XX}$ is true for any 
transversal morphism of log smooth modulus pairs.
\end{theo}

\begin{proof}
By Lemma~\ref{lemm:2} it suffices to find an open Zariski cover $\{U_i \to \ol{X}\}_{i \in I}$ and show $(SBU)_{\ZZ_J \stackrel{}{\to} \XX_J}$ is true for each $J \subsetneq I$ where $U_J = \cap_{i \in J} U_i$. Choose a cover as in Corollary~\ref{coro:etNisLoc}. By Lemma~\ref{lemm:3}, it suffices to show that $(SBU)_{\ZZ \otimes \{0, \dots, 0\} \to \ZZ\otimes(\PP^c, \PP^{c-1})}$ is true, where $\ZZ=(\ol{Z},\Zinf)$ and $\{0, \dots, 0\} \in \PP^c$ is the point orthogonal to $\PP^{c-1} \subseteq \PP^c$. But by Lemma~\ref{lemm:1}, 
\[ M(\ZZ) \to M((\PP^c, \PP^{c-1})\otimes\ZZ) \]
is an isomorphism, and so is 
\[ M((E, \varnothing) \otimes\ZZ) \to M((P, H) \otimes \ZZ), \] 
where $P \to \PP^c$ is the blowup with centre $\{0, \dots, 0\}$, exceptional divisor $E$, and strict transform $H$ of $\PP^{c-1}$. Hence, in this case, the two horizontal morphisms in the square \eqref{blowupsquare} become isomorphisms in $\uMDM^{\eff}$, or in other words, the rows become zero objects. Hence, their cone, the object associated to the square, is also the zero object.
\end{proof}

\section{Some homological algebra} \label{sec:homAlg}

\begin{lemm} \label{lemm:2o3}
Consider a commutative diagram
\[ \xymatrix{
A \ar[r] \ar[d]_a & B \ar[r] \ar[d]_b & C \ar[d]_c \\
A' \ar[r]  & B' \ar[r]  & C'  \\
} \]
in an additive category $\cA$, the associated complexes
\[ A \to A' \oplus B \to B', \]
\[ A \to A' \oplus C \to C', \]
\[ B \to B' \oplus C \to C', \]
and suppose we have a triangulated functor $\Phi: K^b(\cA) \to T$ to some triangulated category $T$. Then two of the above complexes are zero in $T$ if and only if the third one is also zero.
\end{lemm}

\begin{proof}
First note that the complex $A \to A' \oplus B \to B'$ is zero if and only if $Cone(a) \to Cone(b)$ is an isomorphism, since the former is the cone of the latter. The analogous statement is true for the other two complexes. 
Then notice that two of the three morphisms $\xymatrix{ Cone(a) \ar@/^9pt/[rr] \ar[r] & Cone(b) \ar[r] & Cone(c)}$ are quasi-isomorphisms of and only if the third is.
\end{proof}

%
%


\end{document}